\documentclass[12pt]{amsart}
\usepackage{pdfsync}
\usepackage{amscd}
\usepackage{amssymb}
\usepackage{enumitem}
\evensidemargin=\oddsidemargin
\setenumerate{leftmargin=*, itemindent=0pt}
\newcommand\R{\mathbf R}

\newcommand\bH{\mathbf H}
\newcommand\bCa{{\mathbf O}}
\newcommand\LA{{\mathsf A}}

\newcommand\LG{{\mathsf G}}

\newcommand\eS{\rm S}
\newcommand\g{{\mathfrak{g}}}
\newcommand\h{{\mathfrak{h}}}

\newcommand\m{{\mathfrak{m}}}
\newcommand\p{{\mathfrak{p}}}
\newcommand\princ{\rm pr}
\newcommand\s{{\mathfrak{s}}}

\newcommand\op{\oplus}

\newcommand\so{{\mathfrak{so}}}
\newcommand\su{{\mathfrak{su}}}
\renewcommand\u{{\mathfrak{u}}}
\renewcommand\k{{\mathfrak{k}}}
\renewcommand\t{{\mathfrak{t}}}

\newcommand\Ad{\operatorname{Ad}}

\newcommand\SO{\operatorname{\sf SO}}
\newcommand\Sp{\operatorname{\sf Sp}}
\newcommand\SU{\operatorname{\sf SU}}
\newcommand\SUxU[2]{ {\rm {\sf S}({\sf U}(}#1){\rm {\sf U}(}#2{\rm ))} }

\newcommand\U{\operatorname{\sf U}}

\newcommand\Aut{\operatorname{Aut}}
\newcommand\Id{\operatorname{Id}}

\theoremstyle{plain}
\newtheorem{theorem}{Theorem}
\newtheorem*{Main}{Main Theorem}

\newtheorem{lemma}[theorem]{Lemma}
\newtheorem{proposition}[theorem]{Proposition}

\theoremstyle{definition}

\newtheorem{notation}[theorem]{Notation}
\newtheorem{example}[theorem]{Example}

\theoremstyle{remark}
\newtheorem{remark}[theorem]{Remark}

\numberwithin{equation}{section}
\numberwithin{theorem}{section}

\begin{document}
\title[Nonnegatively curved metrics]
{Nonnegatively curved homogeneous metrics in low dimensions}
\subjclass[2000]{Primary: 53C30; Secondary: 53C21; 57S15}
\author[M.\ M.\ Kerr]{Megan M.\ Kerr}
\address{Department of Mathematics, Wellesley College}
\email{mkerr@wellesley.edu}
\author[A.\ Kollross]{Andreas Kollross}
\address{Institut f\"{u}r Geometrie und Topologie, Universit\"{a}t Stuttgart}
\email{kollross@mathematik.uni-stuttgart.de}
\date{\today}

\begin{abstract}
We consider invariant Riemannian metrics on compact homogeneous spaces $G/H$
where an intermediate subgroup~$K$ between $G$ and~$H$ exists. In this case,
the homogeneous space $G/H$ is the total space of a Riemannian submersion. The
metrics constructed by shrinking the fibers in this way can be interpreted as
metrics obtained from a Cheeger deformation and are thus well known to be
nonnegatively curved. On the other hand, if the fibers are homothetically
enlarged, it depends on the triple of groups $(H,K,G)$ whether nonnegative
curvature is maintained for small deformations.

Building on the work of L.~Schwachh\"ofer and K.~Tapp \cite{ST},  we examine
all $G$-invariant fibration metrics on $G/H$ for $G$ a compact simple Lie group
of dimension up to 15. An analysis of the low dimensional examples provides
insight into the algebraic criteria that yield continuous families of
nonnegative sectional curvature.
\end{abstract}
\maketitle

\section{Introduction}
The study of Riemannian manifolds of nonnegative or positive sectional
curvature is one of the original questions of global Riemannian geometry. This
is an area of geometry that has motivated deep and beautiful mathematical
results, and  is still characterized more by its open questions than by its
known theorems.

We know few obstructions to nonnegative or positive curvature. There are some
topological restrictions, most famously  Bonnet-Myers and Synge's Theorem.
There are few known examples of spaces with positive curvature.  When we relax
the curvature condition to nonnegative curvature, we get more examples of
manifolds with nonnegative sectional curvature, many of these discovered within
the past ten years \cite{Wi2},\cite{Z1}.

All known examples of compact irreducible manifolds of nonnegative curvature,
homogeneous and inhomogeneous, come from constructions that begin with a
compact Lie group and a biinvariant metric. Riemannian submersion metrics are a
natural starting point, since, by O'Neill's formula, we know that taking a
quotient tends to increase sectional curvature. In light of the prevalence of
quotients of Lie group actions, it makes sense to fully understand the basic
setting.  In this paper we consider a simple deformation of homogeneous metrics
as our source for more examples of spaces of nonnegative sectional curvature.

To begin, we take a fibration of homogeneous spaces arising from a chain
$(H,K,G)$ of nested compact Lie groups $H \subsetneq K \subsetneq G$:
\begin{equation*}
\begin{array}{ccc}
K/H & \longrightarrow  & G/H \\
~ & ~ & \downarrow  \\
~ & ~ & G/K. \end{array}
\end{equation*}
Let $\g$ be the Lie algebra corresponding to $G$ and let $\k$ and $\h$ be the
Lie subalgebras corresponding to the subgroups $K$ and $H$ of~$G$,
respectively, such that  $\h \subset \k \subset \g$. Let $g_0$ be a biinvariant
metric  on $G$.  We use $g_0$ to fix orthogonal complements: Let $\m$ be the
subspace of~$\k$ orthogonal to~$\h$ such that $\k =\h \op \m$, the tangent
space in the direction of the fiber $K/H$, and let $\s$ be the subspace of~$\g$
orthogonal to~$\s$ such that $\g = \k \op \s$, the tangent space in the
direction of the base $G/K$. Notice,  $\g = \h \op \m \op \s$ and the tangent
space to the total space $G/H$ is $\p = \m \op \s$.  For any element $X$ in
$\p$, we may write $X =X^{\m} + X^{\s}$, where $X^{\m}$ in $\m$ denotes the
vertical component of~$X$, and $X^{\s}$ in $\s$ denotes the horizontal
component.
We define the following one-parameter family of metrics on $G/H$:
\begin{equation}\label{deformation}
g_t (X,Y) = \frac{1}{1-t} \cdot g_0(X^{\m},Y^{\m}) + g_0(X^{\s},Y^{\s}).
\end{equation}

\begin{Main}\label{main}
Let $G$ be a simple compact Lie group of dimension at most~15. Then the
homogeneous space $G/H$ with fibration metric $g_t$ corresponding to a chain
$(H,K,G)$ of nested compact Lie groups admits nonnegative sectional curvature
for small $t > 0$ if and only if one of the following holds:
\begin{itemize} \item[(i)]  $(K, H)$ is a symmetric pair, or more generally, $[\m,\m]^\m =0$;
\item[(ii)] the chain $(H,K,G)$ is one of $(\SU(2),\SO(4),\SO(5))$ or
    $(\SU(2),\SO(4),\LG_2)$ where in the second case the subgroup $\SU(2)$
    is such that $\SU(2)\subset\SU(3)\subset\LG_2$.
\end{itemize}
\end{Main}
Our result explores and builds on the following result of Schwachh\"ofer and
Tapp, who also showed that the chain $(\SU(2),\SO(4),\LG_2)$ satisfies the
inequality~($*$) below.
\begin{theorem}{\rm \cite{ST}}\label{ThmST}
(1) The metric $g_t$ has nonnegative curvature for small $t > 0$ if and only if
there exists some $C > 0$ such that for all $X$ and $Y$ in $\p$,
\begin{equation}\label{condition}
|[X^{\m}, Y^{\m}]^{\m}| \leq C|[X,Y]|.\tag{$*$}
\end{equation}
(2) In particular, if $(K, H)$ is a symmetric pair, then  $g_t$ has nonnegative
curvature for small $t > 0$, and in fact for all $t \in (-\infty, 1/4]$.
\end{theorem}

\begin{remark}\label{trivial-Ks}
The first part of~(2) above follows from the observation that when $(K,H)$ is a
symmetric pair, we have $[\m,\m] \subseteq \h$. In this case the left hand side
of the inequality in condition~(\ref{condition}) is always zero. Also, note
that whenever $[\m,\m]=0$, as when $K$ is a torus, condition~(\ref{condition})
holds. We say that $H$ is a \emph{symmetric subgroup} of~$K$ if we have
$[\m,\m] \subseteq \h$.
\end{remark}

\begin{remark}
When $H$ is trivial, our fibration is
\begin{equation*}
K \to G \to G/K
\end{equation*}
and $g_t$ is in fact a {\sl left-invariant} metric on~$G$.  In this case,
Schwachh\"ofer proved $g_t$ has nonnegative curvature for small $t > 0$ only if
the semisimple part of~$\k$ is an ideal of~$\g$, see~\cite{Sch}. In particular,
when $\g$ is simple and $\k$ is nonabelian, one does not get nonnegative
curvature.
\end{remark}

Condition~(\ref{condition}) follows if the chain $(H,K,G)$ satisfies the
hypothesis of the following Theorem of Schwachh\"{o}fer and Tapp.  In this case,
more arbitrary changes of the metric on $G / H$ preserve nonnegative curvature.

\begin{theorem}{\rm \cite{ST}}\label{ThmST2}
If there exists $C > 0$ such that for all $X, Y \in \p$,
\begin{equation}\label{condition2}
|X^{\m} \wedge Y^{\m}| \le  C\, |[X, Y ]|,\tag{$**$}
\end{equation}
then any left-invariant metric on~$G$ sufficiently close to~$g_0$ which is
$\Ad_H$-invariant and is a constant multiple of~$g_0$ on~$\s$ and $\h$ (but
arbitrary on $\m$) has nonnegative sectional curvature on all planes contained
in~$\p$; hence, the induced metric on $G/H$ has nonnegative sectional
curvature.
\end{theorem}

We are interested in the class of metrics on $G/H$ obtained via scaling along
the fibration above.  The goal is a further exploration of the algebraic
sources of condition~(\ref{condition})
to find new spaces with metrics of nonnegative curvature within this framework.
In this paper we consider all chains where $G$ is a low-dimensional simple
compact Lie group.  In all cases where $(K,H)$ is not a symmetric pair, we find
$X=X^\m + X^\s$ and $Y=Y^\m + Y^\s$ in $\p=\m \op \s$ for which  $[X,Y]=0$ yet
$[X^{\m},Y^{\m}]^{\m} \neq 0$; hence, condition~(\ref{condition}) fails. We do
not know, however, whether the failure of condition~($*$) is equivalent to the
existence of such a pair of vectors~$X,Y$.

It is clear that whenever there is such a pair of vectors $X,Y \in \m\op\s$ for
some chain $(H,K,G)$, then condition~(\ref{condition}) fails as well for  the
chain $(H', K, G')$, for every closed subgroup $H' \subset H$, and for every
$G'$ such that $G \subset G'$, cf.~Lemma~\ref{LemmaHprime}.

\begin{example}
To illustrate the delicate relationship here between the algebra and the
geometry, we give the following pair of fibrations.  In $\LG_2$, there are two
subgroups in $\SO(4)$, each isomorphic to $\SU( 2)$, but not conjugate in
$\LG_2$.  There are  two chains $(H, \SO(4), \LG_2)$ with $H \cong \SU(2)$. We
let $\widetilde{\SU(2)}$ denote the $\SU(2)$ which is \emph{not} a subgroup of
$\SU(3)$ in $\LG_2$, cf.~Table~\ref{TSubgroupsOfG2}.  The chain
$(\SU(2),\SO(4),\LG_2)$ satisfies condition~(\ref{condition}), while the  chain
$(\widetilde{\SU(2)},\SO(4),\LG_2)$ does not.  Notice in both of these
examples, the base is $\LG_2/\SO(4)$ and the fibers are real projective
spaces~$\R{\rm P}^3$ with a symmetric metric.
This pair of fibrations shows that a complete understanding of the algebraic
criteria which govern the geometry here must go beyond notions like
subalgebras, ideals, rank, etc.
\end{example}


\section{Low-dimensional examples}
In this section we analyze all simple compact Lie groups~$G$ of dimension up to
15. For each Lie group $G$, we completely determine the chains of closed
connected subgroups $H \subsetneq K \subsetneq G$ for which
condition~(\ref{condition}) holds.  Thus, we find all those homogeneous spaces
$G/H$ admitting a one-parameter family of fibration metrics with nonnegative
curvature.

For each Lie group $G$, whenever $K$ is a torus, condition~(\ref{condition})
holds trivially, by Remark~\ref{trivial-Ks}.  Thus we do not need to further
consider the tori $K$ in $G$. Similarly, condition~(\ref{condition}) holds when
$H$ is a \emph{symmetric subgroup} of~$K$. Each Lie group $G$ we consider here
is simple,  thus when $K$ is nonabelian, we will not need to consider the case
of trivial $H$.  Recall that when $H$ is trivial, our fibration metrics are in
fact left-invariant metrics on $G$. In \cite{Sch}, Schwachh\"ofer proves they
have nonnegative curvature only if the semi-simple part of~$\k$ is an ideal
of~$\g$.

In what follows, we will need the following lemmas, included here for
convenience.

\begin{lemma}\label{LemmaBrInt}
Let $(H,K,G)$ be a chain of compact groups such that two elements of~$\m$
commute if and only if they are linearly dependent.  If
\[
\{[U,V] \mid  U, V \in \m\}\cap \overline{\{[W,Z]^{\k} \mid W, Z \in \s\}} =
\{0\},
\]
then condition~(\ref{condition2}) holds. (Here $\overline{S}$ denotes the
topological closure of a subset $S$.)
\end{lemma}

\begin{proof}
This is a reformulation of the method that is used to prove
\cite[Proposition~4.2]{ST}. For the convenience of the reader we reproduce the
proof of Schwachh\"{o}fer and Tapp here: Suppose condition~(\ref{condition2}) is
not satisfied. Then there exist sequences $\{X_r\}$ and $\{Y_r\}$ in~$\p$ such
that for each $r$ the pair of vectors $(X_r^{\m} , Y_r^{\m})$ is orthonormal
and $\lim [X_r, Y_r] = 0$. After passing to subsequences, we know orthonormal
limits $X^{\m} := \lim X_r^{\m}$ and $Y^{\m} := \lim Y_r^{\m}$ exist.  By
hypothesis, $B := [X^{\m}, Y^{\m}] \neq 0$. Meanwhile, $0 = \lim [X_r,
Y_r]^{\k} = \lim [ X_r^{\m},Y_r^{\m} ] + \lim [ X_r^{\s}, Y_r^{\s} ]^{\k}$, so
that $B = -\lim[ X_r^{\s}, Y_r^{\s} ]^{\k}$ is a nonzero element of $\{ [U,V]
\mid U, V \in \m\}\cap \overline{\{[W,Z]^{\k} \mid W, Z \in \s\}}$. This yields
our contradiction. Thus condition~(\ref{condition2}) must hold.
\end{proof}

\begin{lemma}\label{LemmaHprime}
Let $H' \subseteq H \subsetneq K \subsetneq G \subseteq G'$ be a chain of
nested compact groups such that condition~(\ref{condition}) cannot hold for the
chain $(H,K,G)$. Then condition~(\ref{condition}) cannot hold for the chain
$(H', K,G')$.
\end{lemma}

\begin{proof}
Let $\g = \h \oplus \m \oplus \s$ such that $\m$ is the orthogonal complement
of~$\h$ in~$\k$ and $\s$ is the orthogonal complement of~$\k$ in~$\g$. By
hypothesis, we can find sequences $\{X_r\}$ and $\{Y_r\}$ in~$\p = \m \op \s$
such that $|[X_r^{\m} , Y_r^{\m}]^{\m}| = 1$ and $\lim [X_r,Y_r] = 0$. Now let
$\g' = \h' \op \m' \op \s'$ be the decomposition of~$\g'$ such that $\m'$ is
the orthogonal complement of~$\h'$ in~$\k$ and $\s'$ is the orthogonal
complement of~$\k$ in $\g'$. We have $\m \subseteq \m'$ and $\s \subseteq \s'$.
Then $\p \subseteq \p' = \m' \op \s'$. The same sequences $\{X_r\}$ and
$\{Y_r\}$ are in~$\p'$ and we have
\[
|[X_r^{\m'} , Y_r^{\m'}]^{\m'}|
 = |[X_r^{\m} , Y_r^{\m}]^{\m'}|
 \ge \left|[X_r^{\m} , Y_r^{\m}]^{\m}\right| = 1.
\]
Thus condition~(\ref{condition}) is not satisfied for the chain $(H',K,G')$.
\end{proof}

\begin{notation}
By $\LA_1^{\princ}$, $\SO(3)^{\princ}$, or $\SU(2)^{\princ}$, we denote
subgroups of certain simple Lie groups which correspond to \emph{principal
three-dimensional subalgebras} see~\cite[\S9]{dynkin1}.
\end{notation}

 \begin{notation}
 When we give specific choices of~$X$ and $Y$ in $\p$, we will use the
 following standard Lie algebra notation. Let  $E_{ij}$ ($i\neq j$) denote the
 skew-symmetric matrix with $ij^{th}$ entry 1 and $ji^{th}$ entry $-1$, all
 other entries 0. Let  $F_{ij}$ ($i\neq j$) denote the symmetric matrix with
 $ij^{th}$ and $ji^{th}$ entries 1, all other entries 0. Let $F_{jj}$ denote
 the diagonal matrix with $jj^{th}$ entry 1, all other entries 0.
 \end{notation}

\subsection{  \boldmath $G = \SU(2)$}
We note that $\SU(2)$ has no nonabelian subgroups $K$. Since the biinvariant
metric on $G = \SU(2) = {\sf S}^3$ has positive curvature,
 it is no surprise that the only  chain, $(\Id,\SO(2),\SU(2))$,  fulfills
condition~(\ref{condition}). In fact, all left-invariant metrics with
nonnegative sectional curvature on $\SU(2)$ are classified in \cite{BFSTW}.


\subsection{ \boldmath $G = \SU(3)$ }
There are only three conjugacy classes of non-abelian subgroups of~$\SU(3)$,
which we will denote as ${\sf S}(\U(1)\times \U(2))$, $\SU(2)$ and $\SO(3)$.

\begin{enumerate}

\item {\boldmath $K = {\sf S}(\U(2)\times \U(1))$.} All closed subgroups $H
    \subset K$ of rank~2 are symmetric.

\begin{enumerate}

\item The subgroup $H=\SU(2) \subset K$ has a one-dimensional
    complement $\m$. Since $[\m,\m]=0$,  condition~(\ref{condition}) is
    trivially satisfied.

\item For the real circle subgroup $H=\SO(2) \subset \SU(2) \subset K$,
    we see $[\m,\m] \subseteq \h$. Condition~(\ref{condition}) again is
    trivially satisfied.

\item Consider the one-parameter family of  circles
    $H=\Delta_{p,q}\U(1) \subset T^2 \subset K=\U(2)$.  Recall, the Lie
    subalgebra is (for coprime integers
$p$ and $q$)
\[
\qquad \qquad \Delta_{p,q}\u(1) := {\rm span}\left\{\left(\begin{array}{ccc}
pi &  0 & 0  \\
0  & qi  &  0 \\
0  & 0  &   -(p+q)i  \end{array}\right) = i(pF_{11} +qF_{22}
-(p+q)F_{33})\right\}.
\]
Note that  $\Delta_{(p,q)}\U(1)$ and $\Delta_{(-p,-q)}\U(1)$ denote the
same subgroup; thus, without loss of generality we will always take
$p\geq 0$. We find that  condition~(\ref{condition}) is fulfilled if
and only if $(p,q)=(1,-1)$. When $(p,q)=(1,-1)$,  one sees that this
circle is a symmetric subgroup: $[\m,\m] \subseteq \h =
\Delta_{(1,-1)}\U(1)$. Thus condition~(\ref{condition}) is  fulfilled.
Otherwise, we may take $X^{\m} = E_{12}$, $Y^{\m} = iF_{12}$ (in $\m$
for all $(p,q)$) and $X^{\s} = E_{13}+E_{23}$, $Y^{\s} =
i(F_{23}-F_{13})$ in $\s$. We see that $[X^{\m} +X^{\s},Y^{\m} +
Y^{\s}]=0$, while $[X^{\m},Y^{\m}] = 2i(F_{11}-F_{22})$. Provided
$(p,q) \neq (1,-1)$, this has a nonzero $\m$-component.
\end{enumerate}

\item{\boldmath$K = \SU(2)$ and $K=\SO(3)$.} One-dimensional subgroups of
    $\SU(2)$ and $\SO(3)$ are symmetric.

\end{enumerate}

\begin{remark}\label{K=A_1}
Whenever $K=\SU(2)$ or $K=\SO(3)$, the only nontrivial subgroups $H$ are
circles, which are symmetric. Furthermore, whenever $K$ is abelian,
$[\m,\m]=0$, and hence condition~(\ref{condition}) always holds.
\end{remark}


\subsection{ \boldmath $G = \SO(5)$}

In Table~\ref{TSubgroupsOfB2} we give the conjugacy classes of closed
nonabelian connected subgroups of $\SO(5)$, together with inclusion relations.
Note that the two normal subgroups of~$\SO(4) = \SU(2) \cdot \SU(2)$ are
conjugate by an inner automorphism of~$\SO(5)$.

\begin{table}[h]
\begin{picture}(100,100)
\put(0,40){\fontsize{10}{14}
\begin{tabular}{cccccc}
 & & $\SO(5)$ &&& \\
 & & &&& \\
 $\SO(4)$ & & $\SO(3) {\times} \SO(2)$ &&& $\LA_1^{\princ}$ \\
 & & &&& \\
 $\U(2)$ & & $\SO(3)$ &&& \\
 & & &&& \\
 $\SU(2)$ & & &&& \\
 & & &&& \\
\end{tabular}}
\put(20,16){\line(0,1){11}} \put(20,44){\line(0,1){13}}
\put(85,44){\line(0,1){13}} \put(85,70){\line(0,1){13}}
\put(148,72){\line(-6,1){40}} \put(62,45){\line(-3,1){28}}
\put(24,72){\line(4,1){32}}
\end{picture}\hspace{5em}
\caption{Conjugacy classes of nonabelian connected subgroups in~$\SO(5)$}
\label{TSubgroupsOfB2}
\end{table}

\begin{enumerate}

\item{\boldmath$K = \SO(4)$.}  Here \[\SO(4) = \begin{pmatrix} 1 & 0
    \\ 0 & \SO(4) \end{pmatrix} \subset \SO(5).\]
All nonabelian closed subgroups of $\SO(4)$ are symmetric except $\SU(2)$.

\begin{enumerate}

\item Let us consider the chain $(\SU(2),\SO(4),\SO(5))$. In this case,
    $\m \subset \g$ is a subalgebra isomorphic to~$\su(2)$. Since all
    nonzero matrices in $\su(2)$ are invertible, all nonzero matrices
    in the set of brackets $[U,V]$ where $U,V \in \m$ are of rank~4. On
    the other hand, for elements $W,Z \in \s$, brackets $[W,Z] =
    [W,Z]^{\k}$ have rank at most~2. Thus we conclude that the chain
    $(\SU(2),\SO(4),\SO(5))$ fulfills the hypothesis of
    Lemma~\ref{LemmaBrInt}, and hence of Theorem~\ref{ThmST2}.  Indeed,
    this chain here corresponds on the Lie algebra level to the chain
    $(\Sp(1), \Sp(1) \times \Sp(1), \Sp(2))$, which is exactly the Hopf
    fibration $ \eS^7 \to \bH{\rm P}^1 = \eS^4. $ Since $\eS^7$ with
    the normal homogeneous metric has positive sectional curvature, so
    has any sufficiently small deformation.

\item The subgroup $H=T^2 \subset \SO(4)$ is symmetric.

\item Consider the one-parameter family of diagonal circles
    $H=\Delta_{\theta} \SO(2) \subset T^2 \subset K=\SO(4)$. On the Lie
    algebra level,
\[
\h = \rm{span}\{\cos\theta \,E_{23} + \sin\theta
\,E_{45}\} \subset \rm{span}\{E_{23}, E_{45}\} =\t.
\]
We show that condition~(\ref{condition}) fails for all $\theta$. If
$\sin\theta \neq 0$, we may take $X^{\m} = E_{34}$, $Y^{\m} = E_{24}$,
which lies in $\m$ for all $\theta$, $X^{\s} =E_{13}$, $Y^{\s} =
-E_{12}$. We see $[X^{\m}, Y^{\m}] = E_{23}$, which has a nonzero
$\m$-component, yet  $[X,Y] = [X^{\m} + X^{\s}, Y^{\m} + Y^{\s}] = 0.$
If instead $\sin\theta = 0$, we  choose $X^{\m} = E_{25}$, $Y^{\m} =
E_{24}$ (again, in $\m$ for all $\theta$) and $X^{\s} =E_{14}$, $Y^{\s}
=E_{15}$. We  have $[X^{\m}, Y^{\m}] = E_{45}$, which also has a
nonzero $\m$-component, while $[X,Y]=0$. Thus
condition~(\ref{condition}) fails.

\end{enumerate}

\item{\boldmath$K = \SO(3) \times \SO(2)$.} Here \[\SO(3)\SO(2) =
    \begin{pmatrix}
    \SO(3) & 0
    \\ 0 & \SO(2) \end{pmatrix} \subset \SO(5).\]
\begin{enumerate}

\item For the subgroups $H=\SO(3)\times \Id$, $H=\SO(2)\times\SO(2)$,
    and $H=\SO(2) \times \Id$ of~$K$ we have $[\m,\m]^{\m}=0$.

\item Consider the one-parameter family of one-dimensional subalgebras
\[
\h = \rm{span}\{\cos\theta \,E_{13} + \sin\theta
\,E_{45}\} \subset \rm{span}\{E_{13}, E_{45}\} = \t.
\]
When $\sin\theta =0$ we have again $H=\SO(2)\times \Id$, with
$[\m,\m]\subseteq \h$.  But when $\sin\theta\neq 0$, we no longer have
that $(K,H)$ is a symmetric pair and we show
condition~(\ref{condition}) is not fulfilled. We take the following $X$
and $Y$ in $\p$: $X^{\m} = E_{13}$, $Y^{\m} = E_{23}$, and $X^{\s} =
-E_{14}$, $Y^{\s} =E_{24}$. Then $[X^{\m},Y^{\m}] =
[E_{13},E_{23}]=-E_{12}$ has a nonzero $\m$-component when
$\sin\theta\neq 0$, yet $[X,Y] =0.$ Thus condition~(\ref{condition})
fails.

\end{enumerate}

\item {\boldmath$K = \U(2)$.} Here \[\U(2) = \left\{\begin{pmatrix} A & -B
    \\ B & A \end{pmatrix} \in \SO(4) \mid A + iB \in\U(2)\right\} \subset \SO(4) \subset
    \SO(5).\]
\begin{enumerate}

\item The subgroups $H=\SU(2)$ and $H=T^2$ are symmetric.

\item We have a one-parameter family of circles
\[
H=\Delta_{p,q}\U(1) \subset T^2 \subset K = \U(2) \subset \SO(5).
\]
As in Example 2.2(1c), we always take $p\geq 0$. When $(p,q)=(1,-1)$,
$H=\Delta_{1,-1}\U(1) = \SUxU11$ is a symmetric subgroup of~$K$ and
$[\m,\m] \subset \h$. Hence, condition~(\ref{condition}) holds
trivially in this case. We show that condition~(\ref{condition}) fails
in case $(p,q) \neq (1,-1)$. We have $\s = \so(5)\ominus\u(2)$ and $\m
=\u(2) \ominus \u_{p,q}(1)$ (note that $\m \neq \su(2)$). When $(p,q)
\neq (1,-1)$, we exhibit a commuting pair. We take
\[
X^{\m} = \tfrac12(E_{25} +E_{34}) \quad \hbox{and} \quad
Y^{\m}=\tfrac12(E_{23}+E_{45}),
\]
these elements of $\u(2) \ominus \t^2$ are in $\m$ for any pair
$(p,q)$. Furthermore, we take
\[
X^{\s} = E_{14} +\tfrac12(E_{23}-E_{45}) \quad \hbox{and} \quad Y^{\s}
=  E_{12} +\tfrac12(E_{25}-E_{34}).
\]
We see that $[X,Y] = [X^{\m} + X^{\s}, Y^{\m} + Y^{\s}] = 0$ while
$[X^{\m},Y^{\m}]  = -\tfrac12(E_{24}-E_{35})$, which has nonzero
$\m$-component provided $(p,q) \neq (1,-1)$.

\end{enumerate}

\item {\boldmath $K = \SU(2)$, $K=\SO(3)$, $K=\SO(3)^{\princ}$.}
    Condition~(\ref{condition}) is always satisfied, see
    Remark~\ref{K=A_1}.

\end{enumerate}


\subsection{\boldmath $G = \LG_2$}\label{G2Chains}
We have $\LG_2 \cong \Aut(\bCa)$, the automorphism group of the octonions. We
use the usual embedding of~$\LG_2$ in $\SO(7)$ given by the action on the
purely imaginary octonions, which can be described by the basis below for the
Lie algebra (note there is no $X_3$ or $Y_3$).
\begin{align*}
X_1 &= E_{46} - E_{57},
& Y_1 = 2E_{13} + E_{46} + E_{57}, \\
X_2 &= E_{45} + E_{67},
& Y_2 = 2E_{23} - E_{45} + E_{67}, \\
X_4 &= E_{16} + E_{25},
& Y_4 = 2E_{34} - E_{25} +E_{16}, \\
X_5 &= E_{17} -E_{24},
& Y_5 = 2E_{35} + E_{24} + E_{17}, \\
X_6 &= E_{14} + E_{27},
& Y_6 = 2E_{36} + E_{27} - E_{14}, \\
X_7 &= E_{15} - E_{26},
& Y_7 = 2E_{37} - E_{26} - E_{15}, \\
Z_1 &= 2E_{12}  - E_{47}+ E_{56},
& Z_2 = E_{47} + E_{56}.\hspace{3.2em}
\end{align*}

The conjugacy classes of nonabelian compact connected subgroups of~$\LG_2$
together with inclusion relations are given by Table~\ref{TSubgroupsOfG2},
cf.~\cite[Proposition~15]{kollross}. Note that the two normal subgroups of type
$\LA_1$ in $\SO(4)$ in $ \LG_2$ are not conjugate in $\LG_2$; we will
distinguish these nonconjugate isomorphic subgroups of $\SO(4)$ by writing one
of them with a tilde.
\begin{table}[h!]
\begin{picture}(130,100)
\put(0,40){\fontsize{10}{14}
\begin{tabular}{cccccc}
 & & $\LG_2$ & & & \\
 & & & & & \\
 $\SU(3)$ & & $\SO(4)$ & & & $\LA_1^{\princ}$ \\
 & & & & & \\
 $\U(2)$ & & $\SO(3)$ & &  $\widetilde{\U(2)}$ & \\
 & & & & & \\
 $\SU(2)$ & & & &  $\widetilde{\SU(2)}$ & \\
 & & & & & \\
\end{tabular}}
\put(20,16){\line(0,1){14}} \put(20,44){\line(0,1){13}}
\put(68,44){\line(0,1){14}} \put(68,72){\line(0,1){10}}
\put(112,22){\line(0,1){8}} \put(107,49){\line(-4,1){32}}
\put(135,71){\line(-6,1){55}} \put(62,45){\line(-3,1){35}}
\put(24,72){\line(4,1){32}} \put(27,45){\line(3,1){35}}
\end{picture}
\caption{Conjugacy classes of nonabelian connected subgroups in~$\LG_2$}
\label{TSubgroupsOfG2}
\end{table}

\begin{enumerate}

\item{\boldmath $K = \SU(3)$.}  This can be viewed as the group of elements
    of $\SO(7)$ which fix the purely imaginary octonion given by the third
    canonical basis vector of~$\R^7$.  At the Lie algebra level,  $\su(3)=$
    ${\rm span}\{Z_1,Z_2,X_1,\dots,X_7\}$.

\begin{enumerate}

\item The subgroup $H=\U(2) \cong \SUxU21$ is symmetric.

\item When $H = \SU(2)$ , condition~(\ref{condition}) fails. To see
    this, we use $\su(2)=\text{span}\{X_1,X_2,Z_2\}$ and take $X^{\m} =
    \sqrt{2}X_4 + X_5$, $Y^{\m} = \sqrt{2}X_7 + X_6$, $X^{\s} = -Y_6$,
    $Y^{\s} = Y_5$. Then $[X^{\m},Y^{\m}]=Z_1+3Z_2$ has a nonzero
    $\m$-component, and yet $[X,Y] = [X^{\m}+X^{\s},Y^{\m}+Y^{\s}]=0$.

\item When $H=T^2$ (where $\t^2=$ span$\{Z_1,Z_2\}$),
    condition~(\ref{condition}) fails. We take $X^{\m}= (\sqrt{2} -
    3)X_4 + X_6$, $Y^{\m} = (2\sqrt{2}-1)X_5+ (\sqrt{2}-1)X_7$, and
    $X^{\s} = (1-\sqrt{2})Y_5 + Y_7$, $Y^{\s} = (1-\sqrt{2})Y_6 +Y_4$.
    Then we see $[X^{\m},Y^{\m}]^{\m} = (\sqrt{2}-1)X_2$ and yet
    $[X,Y]= 0$.

\item The subgroup $H=\SO(3)$ is symmetric in $K=\SU(3)$.

\item When $H=\SO(2)$ ($\SO(2) \subset \SU(2)$),
    condition~(\ref{condition}) fails by Lemma~\ref{LemmaHprime}.

\end{enumerate}

\item{\boldmath $K = \SO(4)$.} At the Lie algebra level, we have
\[
\qquad \so(4) = \su(2) \op \widetilde{\su(2)}, \; \hbox{where} \;
\su(2)=\text{span}\{X_1,X_2,Z_2\},\;
\widetilde{\su(2)}=\text{span}\{Y_1,Y_2,Z_1\}.
\]
    Since all rank two subgroups $H$ of $\SO(4)$  are symmetric, we need
    only consider the rank one subgroups of  $\SO(4)$.

\begin{enumerate}

\item The subgroup $H=\SO(3)$ is symmetric in $\SO(4)$.

\item By Theorem~\ref{ThmST2} in \cite{ST},  when $H=\SU(2)$,
    condition~(\ref{condition})
	is satisfied.

\item When $H = \widetilde{\SU(2)}$, we show
    condition~(\ref{condition}) fails. We take $X^{\m}= X_2$,
    $Y^{\m}=X_1$, and $X^{\s}=X_4+X_5$, $Y^{\s}=X_6+X_7$. Then we see
    $[X^{\m},Y^{\m}] = -2Z_2$ and yet $[X,Y]=[X^{\m} + X^{\s},Y^{\m} +
    Y^{\s}] = 0$.

\item We consider the one-parameter family of circles
    $H=\Delta_{\theta}\SO(2)$, where $\theta \in [0,2\pi)$. We show
    that condition~(\ref{condition}) fails for all $\theta$. At the Lie
    algebra level,
    \[
    \Delta_{\theta}\so(2)={\rm span}\{\tfrac{1}{\sqrt{3}}(\cos \theta)Z_1
    + (\sin \theta)Z_2\}.
    \]
    If $\tan\theta \neq \sqrt{3}$  we take $X^{\m}= X_2+Y_2$, $Y^{\m}=
    -3X_1+Y_1$, and $X^{\s} = X_4 + Y_4$, $Y^{\s} = -3X_7+Y_7$. We find
    that $[X,Y]=0$ yet $[X^{\m},Y^{\m}]= 2 (Z_1+3Z_2)$, in $\t^2$. This
    has a nonzero $\m$-component except when $\tan\theta
    =\smash{\sqrt{3}}$. And if $\tan\theta =\smash{\sqrt{3}}$, we can
    instead take $X^{\m} = \smash{\tfrac{1}{\sqrt{2}}(X_1+Y_1)}$,
    $Y^{\m} = \smash{\tfrac{1}{\sqrt{2}}(-X_2+Y_2)}$, while $X^{\s} =
    X_4$ and $Y^{\s} = X_7$. This time $[X,Y]=0$ and $[X^{\m},Y^{\m}]
    =-Z_1-Z_2$. This has a nonzero $\m$-component except when
    $\tan\theta = \frac{1}{\sqrt{3}}$. We conclude that
    condition~(\ref{condition}) fails for every  choice of
    $H=\Delta_{\theta}\SO(2)$.

\end{enumerate}

\item{\boldmath $K = \U(2)$.}  In terms of the basis for $\g_2$ above,
    \[\u(2)={\rm span}\{X_1,X_2,Z_1,Z_2\}.\]

\begin{enumerate}

\item Both $T^2$ and $\SU(2)$ are symmetric subgroups of $\U(2)$.

\item We consider the one-parameter family of circles
    $H=\Delta_{\theta}\U(1) \subset T^2 \subset \U(2)$. Notice this  is
    the same one-parameter family we saw in the case $K=\SO(4)$, but
    since here $K= \U(2)$ is different, $\m$ and $\s$ are also
    different. Here
    \[
    \m = {\rm span }\{X_1,X_2\} \op {\rm span} \left
    \{-\tfrac{1}{\sqrt{3}}(\sin \theta)Z_1 + (\cos \theta)Z_2 \right\}.
    \]
    When $\theta=\frac{\pi}{2}$, $[\m,\m] \subseteq \h$, and
    condition~(\ref{condition}) holds. We show that
    condition~(\ref{condition}) fails if $\theta\neq\frac{\pi}{2}$: We
    may take $X^{\m}=X_2$, $Y^{\m} = X_1 \in\m$ and
    $X^{\s}=-\smash{\tfrac{1}{\sqrt{2}} (X_7-Y_7)}$, $Y^{\s}=
    \smash{\tfrac{1}{\sqrt{2}}(X_4+Y_4)}$. We find that $[X,Y]=0$ and
    $[X^{\m},Y^{\m}]=-2Z_2$, which  has a nonzero $\m$-component when
    $\theta \neq \frac{\pi}{2}$.

\end{enumerate}

\item{\boldmath $K = \widetilde{\U(2)}$.}   In terms of the basis for
    $\g_2$ above,
    \[
    \widetilde{\u(2)}={\rm span}\{Y_1,Y_2,Z_1,Z_2\}.
    \]
\begin{enumerate}
\item As in the case above, both $T^2$ and $\widetilde{\SU(2)}$ are
    symmetric subgroups of $\U(2)$.

\item We again consider the one-parameter family of circles
    $H=\Delta_{\theta}\U(1) \subset T^2$. This time,
\[
\m = {\rm span }\{Y_1,Y_2\} \op {\rm span
}\left\{-\tfrac{1}{\sqrt{3}}(\sin \theta)Z_1 + (\cos
\theta)Z_2\right\}.
\]
We show that condition~(\ref{condition}) holds only for $\sin\theta =
 0$. When $\sin\theta = 0$, $[\m,\m] \subseteq \h$ and
 condition~(\ref{condition}) holds trivially. Otherwise, we take
$X^{\m} = -Y_2$, $Y^{\m} = Y_1$ (in $\m$ regardless of $\theta$), and
$X^{\s} = X_1 +  \sqrt{2}X_6$, $Y^{\s} = X_2+ \sqrt{2}X_5$. For this
pair, $[X,Y]=0$ yet $[X^{\m},Y^{\m}] = -2Z_1$, which has a nonzero
 $\m$-component for $\sin\theta \neq 0$.

\end{enumerate}

\item{\boldmath$K=\SU(2)$, $K=\widetilde{\SU(2)}$, $K=\SO(3)^{\princ}$, and
    $K = \SO(3)$.}     Condition~(\ref{condition}) is always satisfied, see
    Remark~\ref{K=A_1}.

\end{enumerate}

\subsection{  \boldmath $G = \SO(6)$}\label{SO(6)Chains}
We note that at the Lie algebra level,  $\so(6) \cong \su(4)$.  We will use the
$6$-dimensional orthogonal representation in the discussion that follows. We
will need to consider the subgroups of~$\SO(6)$ up to automorphisms.

\begin{table}[h!]
\begin{picture}(156,160)
\put(-75,90) {\fontsize{10}{14}
\begin{tabular}{ccccl}
&  & $\SO(6)$ & & \\
& & & & \\
& $\U(3)$ & $\SO(5)$ & $\SO(4)\SO(2)$ & $\SO(3)\SO(3)$ \\
& & & & \\
$\U(2)\U(1)$ & $\SU(3)$ & $\SO(3)\U(1)$ & $\SO(4)$ & $\SO(3)\SO(2)$ \\
& & & & \\
$\SU(2)\Delta_{\varphi}$ & $\SUxU21$ & $\U(2)$ & & \\
& & & & \\
& $\SU(2)$ & $\Delta\SO(3)$ & $\SO(3)$ & $\SO(3)^{\princ}$ \\
& & & & \\
\end{tabular}}
\put(70,136){\line(0,1){14}} \put(130,136){\line(-4,1){50}}
\put(170,136){\line(-6,1){80}} \put(10,136){\line(4,1){50}}
\put(10,108){\line(0,1){14}} \put(130,108){\line(0,1){14}}
\put(190,108){\line(0,1){14}} \put(128,108){\line(-4,1){50}}
\put(65,108){\line(-4,1){50}} \put(170,108){\line(-6,1){80}}
\put(185,108){\line(-4,1){50}} \put(-45,108){\line(4,1){50}}
\put(38,108){\oval(170,15)[tl]} \put(38,123){\oval(170,15)[br]}
\put(10,80){\line(0,1){14}} \put(5,80){\line(-3,1){40}}
\put(59,80){\line(-6,1){80}} \put(75,80){\line(4,1){50}}
\put(-50,80){\line(-1,5){2}} \put(10,51){\line(0,1){14}}
\put(20,51){\line(5,2){40}} \put(0,51){\line(-5,2){38}}
\put(65,51){\line(-5,6){40}} \put(85,51){\line(0,1){40}}
\put(88,51){\line(4,3){95}} \put(130,51){\line(0,1){42}}
\put(135,51){\line(1,1){42}} \put(168,51){\line(-4,3){92}}
\end{picture}
\vspace{-2em} \caption{Conjugacy classes of nonabelian connected subgroups
in~$\SO(6)$} \label{TSubgroupsOfSO(6)}
\end{table}

Let $\SU(2)\Delta_{\varphi} = \SU(2)\Delta_{\varphi}\U(1)$ (equivalently,
$\SU(2)\Delta_{\varphi}\SO(2)$) denote  the one-parameter family of  subgroups
of $\U(2)\U(1) =\U(3)\cap \SO(4)\SO(2)$ where $\varphi \in[0,2\pi)$ determines
the angle of the diagonal circle. We embed such that  when $\sin\varphi=0$,
$\SU(2)\Delta_\varphi = \U(2)$. In the following, we will assume $\sin\varphi
\neq 0$.

In the table, there are three conjugacy classes of closed subgroups isomorphic
to~$\SO(3)$. They can be distinguished by their representation on~$\R^6$ which
is given by restricting the standard representation of~$\SO(6)$: The group
denoted by $\SO(3)$ acts on $\R^6$ by the adjoint plus a three-dimensional
trivial representation, $\Delta\SO(3)$ acts by the direct sum of two copies of
the adjoint representation and $\SO(3)^{\princ}$ acts irreducibly on a
5-dimensional linear subspace of~$\R^6$.

\begin{proposition}\label{PropSO6Subgr}
All conjugacy classes of closed connected nonabelian subgroups w.r.t.\ inner
and outer automorphisms of~$\SO(6)$ and their inclusion relations are given by
Table~\ref{TSubgroupsOfSO(6)}.
\end{proposition}

\begin{proof}
To verify the table, it suffices to check at each node if the nonabelian
maximal subalgebras are correctly represented, cf.~\cite[Theorem~15.1]{dynkin1}
or \cite[Theorem~2.1]{hyperpolar}.
\end{proof}

\begin{enumerate}
\item{\boldmath $K=\U(3)$.} Here we view
\[
\U(3) = \left\{ \left( \begin{array}{rr} A & -B \\ B & A \end{array} \right)
\in \SO(6) \middle\vert A+iB\in\U(3)\right\}.
\]

\begin{enumerate}

\item Condition~(\ref{condition}) holds if $H$ is one of the subgroups
    $\U(2)\U(1)$, $\SU(3)$, $\SO(3)\U(1)$, $\SUxU21$, and
    $\Delta\SO(3)$, for which we have $[\m,\m] \subseteq \h$.

\item When $H=T^3$, we show condition~(\ref{condition}) fails. Let
    $\t^3 = {\rm span}\{E_{14}, E_{25}, E_{36}\}$. We take
    $X^{\m}=E_{12}+E_{45}$, $Y^{\m}=E_{23}+E_{56}$ and
    $X^{\s}=E_{12}-E_{45}$, $Y^{\s}=-E_{23}+E_{56}$. Then
    $[X^{\m},Y^{\m}]=E_{13} +E_{46}\in \m$, and $[X,Y]=[X^{\m} +
    X^{\s},Y^{\m} + Y^{\s}] = 0$.  This also shows
    condition~(\ref{condition}) fails for every abelian subgroup  $H'
    \subset T^3 \subset K=\U(3)$ by Lemma~\ref{LemmaHprime}.

\item When
\[
 H= \U(2)=\begin{pmatrix} \U(2) & 0 \\ 0 & 1 \end{pmatrix} \subset \U(3),
\]
we show condition~(\ref{condition}) fails.   We can take $X^\m = E_{13}
+ E_{46}$, $Y^\m = E_{36}$, and $X^\s = E_{26}-E_{35}$, $Y^\s =
E_{12}-E_{45}$. Then $[X^\m,Y^\m]  = - [X^\s,Y^\s] = E_{16}+E_{34}$ and
this is in $\m$, while $[X^\m+X^\s,Y^\m+Y^\s]=0$.  This also shows
    condition~(\ref{condition}) fails for  $H =\SU(2)$ by
    Lemma~\ref{LemmaHprime}.

\item When $H= \SU(2)\Delta_{\varphi}\U(1)$,
    condition~(\ref{condition}) fails, provided  $\tan \varphi \neq
    -\sqrt{2}$. If $\tan\varphi = -\sqrt{2}$, then $H=\SUxU21$, a
    symmetric subgroup.  If  $\sin\varphi = 0$, we are in the case (c)
    above. At the Lie algebra level, $\su(2)=$span$\{ E_{12}+E_{45},
E_{15}+E_{24},  E_{14}-E_{25}\}$ and the diagonal $\u(1)$ is
$$\text{span}\{ \tfrac{1}{\sqrt 2}\cos\varphi (E_{14}+E_{25}) +\sin\varphi\, E_{36} \}.$$
We take $X^{\m}=E_{23}+E_{56}$,
    $Y^{\m}=E_{26}+E_{35}$, and $X^{\s}=E_{12}+E_{13}-E_{45}-E_{46}$,
    $Y^{\s}=E_{15}-E_{16}-E_{24}+E_{34}$. Then $[X^\m + X^\s, Y^\m +
    Y^\s]=0$ and $[X^{\m},Y^{\m}]=2(E_{25}-E_{36})$ has a nonzero
    $\m$-component, provided $\tan \varphi \neq -\sqrt{2}$.
\end{enumerate}

\item{\boldmath $K=\SO(5)$.} Here we view
\[
\SO(5) = \begin{pmatrix} \SO(5) & 0 \\ 0 & 1 \end{pmatrix} \subset
\SO(6).
\]

\begin{enumerate}

\item Condition~(\ref{condition}) holds for the symmetric subgroups
    $H=\SO(3)\SO(2)$ and $H=\SO(4)$.

\item When $H=\U(2) \subset \SO(4)$, we show
    condition~(\ref{condition}) fails. We take $X^\m= E_{25}+E_{35}$,
    $Y^\m= E_{15}+E_{45}$, and $X^\s= E_{26}+E_{36}$, $Y^\s= -
    E_{16}-E_{46}$. Then $[X^\m,Y^\m]=E_{12}+E_{13}-E_{24}-E_{34}$ has
    a nonzero $\m$-component, yet $[X,Y]=0$.  Thus for $H=\SU(2), H=T^2
    \subset \U(2)$, condition~(\ref{condition}) fails as well,  by
    Lemma~\ref{LemmaHprime}.

\item When $H=\SO(3) \subset \SO(4)$, we take $X^\m=E_{15}$,
    $Y^\m=E_{14}$, and $X^\s=E_{46}$, $Y^\s=E_{56}$. We see $[X^\m +
    X^\s, Y^\m +Y^\s]=0$  yet $[X^\m, Y^\m]=-E_{45} \in\m$.

\item For the subgroup $H=\SO(3)^{\princ} \subset \SO(5)$, we show
    condition~(\ref{condition}) fails. We give a basis for this maximal
    Lie subalgebra:
\[
\qquad\; \so(3)^{\princ} = \text{span}
\left\{\sqrt3 E_{14} -E_{24}  + E_{35}, \sqrt3 E_{15}+E_{25} +E_{34}, 2E_{23}-E_{45} \right\}.
\]
We take $X^{\m} = E_{12}$,  $Y^{\m} = E_{13}$, and $X^{\s} =  E_{26}$,
$Y^{\s} = -E_{36}$,  so that $[X^{\m},Y^{\m}] =   -E_{23}$ has a
nonzero $\m$ component, yet $[X,Y]=[X^{\m}+X^{\s} ,Y^{\m}+Y^{\s}] = 0.$

\end{enumerate}

\item{\boldmath $K=\SO(4)\SO(2)$.} Here we view
\[
K = \begin{pmatrix}
    \SO(4) & 0
    \\ 0 & \SO(2) \end{pmatrix} \subset \SO(6).
\]

\begin{enumerate}

\item For each of the following subgroups $H$ we have $[\m,\m]\subseteq
    \h$ and thus condition~(\ref{condition}) holds:\\ $\SO(4)$,
    $\SO(3)\SO(2)$, $\SO(3) \subset \SO(4)$, $\U(2)\U(1)$, $\U(2)$,
    $T^3 = \SO(2)\SO(2)\SO(2)$, $T^2 = \SO(2)\SO(2)\cdot\Id$.

\item When $H=\SU(2)\SO(2) \subset \SO(4)\SO(2)$, we show
    condition~(\ref{condition}) fails. Here we use that
    $\so(4)=\su(2)\op\su(2)$, so that one $\su(2)$ factor is in $\h$,
    and $\m$ is the other $\su(2)$ factor: thus, $[\m,\m]=\m$.  (Note
    the two $\SU(2)$ factors are conjugates in $\SO(6)$.) Take
    $X^{\m}=E_{12}-E_{34}$, $Y^{\m}=E_{14}-E_{23}$, and
    $X^{\s}=\sqrt{2}(E_{36}+E_{45})$, $Y^{\s}=\sqrt{2}(E_{16}+E_{25})$.
    Here $\m=\text{span}\{X^{\m},Y^{\m},[X^{\m},Y^{\m}]\}$ and
    $[X,Y]=[X^{\m} + X^{\s},Y^{\m} + Y^{\s}] = 0$.

    \item When $H=\SU(2)\Delta_\varphi$, condition~(\ref{condition})
        fails. At the Lie algebra level, this is
\[\qquad\quad\quad\text{span}\{ E_{12}+E_{34}, E_{14}+E_{23},  E_{13}-E_{24}\}\op
\text{span}\{ \tfrac{1}{\sqrt 2}\cos\varphi (E_{13}+E_{24}) +\sin\varphi\, E_{56} \}.\]
We may take the same $X^\m + X^\s$, $Y^\m + Y^\s$ as in the previous
example. Note that $[X^{\m},Y^{\m}] = -2(E_{13}+E_{24})$ has a non-zero
$\m$-component provided $\sin\varphi \neq 0$.

\item When $H=\Delta_{\theta}\U(1)\SO(2)$ (where $\Delta_{\theta}\U(1)
    \subset \U(2)$) it is easy to see that condition~(\ref{condition})
    fails. When $\U(1)$ is entirely within one of the $\SU(2)$ factors,
    we know condition~(\ref{condition}) fails by the previous example.
    Otherwise, we take $X^{\m}=E_{13}-E_{24}$, $Y^{\m}=E_{23}+E_{14}$
    (in $\m$ regardless of angle $\theta$), and
    $X^{\s}=\sqrt{2}(E_{25}+E_{36})$ and
    $Y^{\s}=\sqrt{2}(E_{15}-E_{46})$. Then
    $[X,Y]=[X^{\m}+X^{\s},Y^{\m}+Y^{\s}]=0$, while
    $[X^{\m},Y^{\m}]=-2(E_{12}+E_{34})$, which has a nonzero
    $\m$-contribution, since $\Delta\U(1)$ does not lie within either
    $\SU(2)$ factor.

\end{enumerate}

\item{\boldmath $K=\SO(3)\SO(3)$.} Here we view
    \[
    K = \begin{pmatrix} \SO(3) & 0
    \\ 0 & \SO(3) \end{pmatrix}\subset \SO(6).
    \]

\begin{enumerate}

\item  Condition~(\ref{condition}) holds for the symmetric subgroups
    $H=\SO(3)\SO(2)$ and $H=\Delta\SO(3)$, as well as for
    $H=\SO(2)\SO(2)$ which has $[\m,\m]\subset \h$.

\item When $H=\SO(3) \cdot \Id\cong\Id \cdot \SO(3)$ we see
    condition~(\ref{condition}) fails. We  use that $\m$ is the other
    subalgebra $\so(3)$, so that $[\m,\m]=\m$. For  $H=\SO(3) \cdot
    \Id$, we may take $X^{\m}=E_{45},Y^{\m}=E_{46}$, so that
    $[X^{\m},Y^{\m}]=-E_{56} \in \m$. By choosing $X^{\s}=E_{25},
    Y^{\s}=-E_{26}$ we get $[X^\m+X^\s,Y^\m+Y^\s]=0$.

\item In the case $H=\Delta_{\theta}\SO(2)$,
    condition~(\ref{condition}) fails. When $\SO(2)$ lies entirely
    within one of the $\SO(3)$ factors, condition~(\ref{condition})
    fails by the previous example.  Otherwise we exhibit a commuting
    pair: $X^{\m}=E_{23}$, $Y^{\m}=E_{13}$, and $X^{\s}=E_{14}$,
    $Y^{\s}= -E_{15}$. We see $[X^\m+X^\s,Y^\m+Y^\s]=0$ while
    $[X^{\m},Y^{\m}]= E_{12}$, which has a nonzero $\m$-component as
    long as our $\SO(2)$ does not lie entirely within either of the
    $\SO(3)$ factors.

\end{enumerate}
\item{\boldmath $K=\U(2)\U(1) = \U(2)\SO(2)$.} (This is a subset of
    $\SO(4)\SO(2)$.)

\begin{enumerate}

\item For each of the subgroups $H=\U(2), T^2\SO(2),
    \SU(2)\Delta_{\varphi}, \SU(2)\SO(2), \SU(2)\cdot\Id$, and
    $\SO(2)\SO(2)$, we have
    $[\m,\m]^\m=0$.  Thus condition~(\ref{condition}) holds in all
    these cases.

\item When $H=\U(1)\cdot\SO(2)$, $\m=\su(2)$ and since $[\m,\m]=\m$, we
    show condition~(\ref{condition}) fails.   We take
    $X^\m=E_{12}+E_{34}$, $Y^\m=E_{14}+E_{23}$, and $X^\s=\sqrt
    2(E_{15}+E_{26})$, $Y^\s=\sqrt 2(E_{35}-E_{46})$. We see
    $[X^\m,Y^\m]=2(E_{13}-E_{24})\in\m$, while $[X^\m +X^\s, Y^\m
    +Y^\s]=0$. By Lemma~\ref{LemmaHprime}, condition~(\ref{condition})
    fails for every abelian $H' \subset \U(1)\cdot\SO(2)$.

\end{enumerate}
\item{\boldmath $K=\SU(3)$.}

\begin{enumerate}

\item Condition~(\ref{condition}) holds for symmetric subgroups
    $H=\SUxU21$, $H=\Delta\SO(3)$.

\item When $H=T^2$ we show condition~(\ref{condition}) fails. We take
    $X^\m=E_{12}+E_{45}$, $Y^\m=E_{13}+E_{46}$, and then we take
    $X^\s=E_{13}-E_{46}$, $Y^\s=E_{12}-E_{45}$. It is easy to see that
    $[X^{\m}+X^{\s},Y^{\m}+Y^{\s}]=0$ while $[X^\m,Y^\m]$ is a nonzero
    element of~$\m$.

\item When $H=\SU(2)$, condition~(\ref{condition}) cannot hold.  We
    exhibit a commuting pair: We take $X^{\m}=E_{23}+E_{56}$,
    $Y^{\m}=E_{26}+E_{35}$, and $X^{\s}=E_{12}+E_{13}-E_{45}-E_{46}$,
    $Y^{\s}=E_{15}-E_{16}-E_{24}+E_{34}$. Then $[X^\m + X^\s, Y^\m +
    Y^\s]=0$ and $[X^{\m},Y^{\m}]=2(E_{25}-E_{36})$ has a nonzero
    $\m$-component.

\end{enumerate}
\item{\boldmath $K=\SO(3)\U(1)$.} (Here, $\SO(3) = \Delta\SO(3) \subset
    \SO(3)\SO(3)$.)

\begin{enumerate}

\item  The subgroup $H=\SO(2)\U(1)$ is symmetric, and
    $H=\Id\cdot\Delta\SO(3)$ has $[\m,\m]=0$.

\item We consider the family of subgroups $H=\Delta_{\theta}\SO(2)$,
    where when $\sin\theta =0$, $H=\U(1)\cdot\Id$, and $\cos\theta=0$
    corresponds to $H=\Id\cdot\SO(2)$. Just as in the case above, when
    $\theta=0$, $[\m,\m]^\m=0$ and condition~(\ref{condition}) holds.
    We prove that when $\sin\theta \neq 0$,
    condition~(\ref{condition}) fails by exhibiting a pair of commuting
    vectors: Let $X^\m=E_{12}+E_{45}$, $Y^\m=E_{13}+E_{46}$, and
    $X^\s=E_{12}-E_{45}$, $Y^\s=E_{13}-E_{46}$. We see $[X^\m +X^\s,
    Y^\m +Y^\s]=0$, while $[X^\m,Y^\m]=-(E_{23}+E_{56})$, which has a
    nonzero $\m$-component, since $\sin\theta \neq 0$.

\end{enumerate}
\item{\boldmath $K=\SO(4)$.}
\begin{enumerate}

\item Both $H=\SO(2)\SO(2)$ and $H=\SO(3)$ are symmetric subgroups.

\item When $H=\SU(2)$ (recall, $\so(4)=\su(2)\op\su(2)$ with conjugate
    factors $\su(2)$), we have $\m=\su(2)$, so that $[\m,\m]=\m$.  In
    this case, condition~(\ref{condition}) fails. We take
    $X^\m=E_{12}-E_{34}$, $Y^\m=E_{13}+E_{24}$, and $X^\s=\sqrt
    2(E_{35}+E_{16})$, $Y^\s=\sqrt 2(E_{25}+E_{46})$. Then $[X^\m +
    X^\s, Y^\m + Y^\s]=0$, while $[X^\m,Y^\m]=2(E_{14}-E_{23})$.

\item When $H=\Delta_{\theta}\SO(2)$, we show
    condition~(\ref{condition}) fails for all choices of~$\theta$. We
    may take $X^\m=E_{23}$, $Y^\m=E_{13}$ and $X^\s=E_{25}$,
    $Y^\s=-E_{15}$. For this pair, $[X^\m + X^\s, Y^\m + Y^\s]=0$ and
    $[X^\m,Y^\m]^\m \neq 0$ as long as $\sin\theta \neq 0$. In case
    $\sin\theta=0$, we may take $X^\m=E_{14}$, $Y^\m=E_{13}$ and
    $X^\s=E_{45}$, $Y^\s=-E_{35}$, so that $[X,Y]=0$ and $[X^\m,Y^\m]
    =E_{34}\in\m$.

\end{enumerate}
\item{\boldmath $K=\SO(3)\SO(2)$.} (This is a subset of $\SO(5)$.)

\begin{enumerate}

\item The subgroup $H=\SO(2)\SO(2)$ is symmetric and $H=\SO(3) \cdot
    \Id$ has $[\m,\m]=0$.

\item We consider the family of subgroups $H=\Delta_{\theta}\SO(2)$,
    where when $\sin\theta =0$, $H=\SO(2)\cdot\Id$, and $\cos\theta=0$
    corresponds to $H=\Id\cdot\SO(2)$. We see that in the case
    $\sin\theta=0$, $[\m,\m]^\m=0$ and condition~(\ref{condition})
    holds. But when $\sin\theta \neq 0$,  we show
    condition~(\ref{condition}) fails:  We take $X^\m=E_{12}$,
    $Y^\m=E_{13}$, $X^\s=E_{24}$, $Y^\s=-E_{34}$.  Here we have $[X^\m
    +X^\s, Y^\m +Y^\s]=0$, while $[X^\m, Y^\m] =-E_{23}$, with a
    nonzero $\m$-component  when $\sin\theta \neq 0$.

\end{enumerate}
\item{\boldmath $K=\SU(2)\Delta_{\varphi}\SO(2)$.}  Recall, $K \subset
    \U(2)\SO(2)$ where $\varphi$ determines the circle angle in
    $\U(2)\SO(2)$. On the Lie algebra level, $\k=\text{span}\{
    E_{12}+E_{34}, E_{14}+E_{23}, E_{13}-E_{24}\}\op \text{span}\{
    \frac{1}{\sqrt 2}\cos\varphi (E_{13}+E_{24}) +\sin\varphi\, E_{56} \}$.

\begin {enumerate}

\item For the symmetric subgroups $H=\SU(2)\cdot\Id$,
    $H=\SO(2)\Delta_{\varphi}\SO(2)$, condition~(\ref{condition}) holds.
\item Consider $H=\Delta_{\theta}\SO(2) \subset
    \SO(2)\Delta_{\varphi}\SO(2)$.  Here the first $\SO(2)$ is the
    symmetric subgroup of $\SU(2)$;  $H$ is the   diagonally embedded
    circle.   When $\sin\theta=0$, $H=\SO(2)\cdot\Id$ and we have
    $[\m,\m]^\m=0$, thus condition~(\ref{condition}) holds. But
when $\sin\theta \neq 0$, then we find a commuting pair. We take
    $X^\m=E_{12}+E_{34}$, $Y^\m=E_{14}+E_{23}$, and $X^\s=\sqrt
 2(E_{15}+E_{26}$, $Y^\s=\sqrt 2(E_{35}-E_{46})$. We have $[X^\m +X^\s,
Y^\m +Y^\s]=0$, while $[X^\m, Y^\m] =2(E_{13}-E_{24})$, which has a nonzero
$\m$-component since $\sin\theta\neq 0$.

\end{enumerate}
\item {\boldmath $K=\SUxU21$.}  (This is a subset of $\SU(3)$.)

\begin {enumerate}
\item  The  subgroups $H=T^3$ and $H=\SU(2)$ are symmetric;
  condition~(\ref{condition}) holds.
  \item Consider the one-parameter family of diagonally embedded circles
      $H=\Delta_{\theta}\SO(2)$. When $\sin\theta=0$, $H=\SO(2) \subset \SU(2)$
      and $[\m,\m]\subset \h$. But when $\sin\theta \neq 0$, we show
      condition~(\ref{condition})  cannot hold, by finding a commuting pair. We
      take $X^\m= E_{12}+E_{45}$, $Y^\m = E_{15} + E_{24}$ (in $\m$ regardless
      of~$\theta$), and $X^\s=\sqrt{2}(E_{16}+E_{35})$,
      $Y^\s=\sqrt{2}(E_{46}-E_{23})$. We have
$[X^\m +X^\s, Y^\m +Y^\s]=0$, while $[X^\m, Y^\m] =2(E_{14}-E_{25})$, which has
a nonzero $\m$-component provided $\sin\theta\neq 0$.
\end{enumerate}
\item {\boldmath $K=\U(2)$.}
\begin{enumerate}
\item Both $T^2$ and $\SU(2)$ are symmetric;   condition~(\ref{condition})
    holds.
\item We consider the one-parameter family of diagonally embedded circles
    $H_{\theta}=\Delta_{\theta}\U(1)$. When $H_{\theta} \subset \SU(2)$,
    then $[\m,\m]^\m=0$. This is the only case when
    condition~(\ref{condition}) holds. Otherwise, the pair of vectors given
    above (in (11)(b)) serves as our commuting pair here with no
    modification needed.
\end{enumerate}
\item {\boldmath $K=\SU(2) \subset \SU(3)$, $K=\Delta\SO(3) \subset
    \SO(3)\SO(3)$,
	$K=\SO(3)\cdot\Id \subset \SO(3)\SO(3)$, $K=\SO(3)^{\princ} \subset
\SO(5)$.}     Condition~(\ref{condition}) is always satisfied, see
Remark~\ref{K=A_1}.

\end{enumerate}



\begin{thebibliography}{9999}

\bibitem[BFSTW]{BFSTW} N.~Brown, R.~Finck, M.~Spencer, K.~Tapp, Z.~Wu, \emph{
    Invariant metrics with nonnegative curvature on compact Lie groups},
    Canad.\ Math.\ Bull.\ \textbf{ 50} (1), 24--34, (2007)

\bibitem[D] {dynkin1} E.B.\ Dynkin: \emph{Semisimple subalgebras of the
    semisimple Lie algebras}, (Russian) Mat.\ Sbornik \textbf{30}, 349-462
    (1952); English translation: Amer.\ Math.\ Soc.\ Transl.\ Ser.~2,
    \textbf{6}, 111-244 (1957)

\bibitem[K1]{hyperpolar} A.\ Kollross: \emph{A classification of hyperpolar and
    cohomogeneity one actions}. Trans.\ Amer.\ Math.\ Soc.\ \textbf{354},
    571--612 (2002)

\bibitem[K2]{kollross} A.~Kollross, \emph{Low Cohomogeneity and Polar Actions
    on Exceptional Compact Lie Groups},  Transformation Groups, \textbf{14}
    (2), 387--415 (2009)

\bibitem[S]{Sch} L.~Schwachh\"ofer, \emph{A remark on left invariant metrics on
    compact Lie groups}, Arch.\ Math.\ \textbf{ 90}, 158--162  (2008)

\bibitem[ST]{ST} L.~Schwachh\"ofer, K.~Tapp, \emph{Homogeneous metrics with
    nonnegative curvature},  J.\ Geom.\ Anal.\  \textbf{19} (4), 929--943
    (2009)

\bibitem[W]{Wi2} B.~Wilking, \emph{Nonnegatively and positively curved
    manifolds}, Metric and Comparison Geometry, Surv.\ Diff.\ Geom.\
    \textbf{11}, ed.\ K.~Grove and J.~Cheeger, International Press (2007)

\bibitem[Z]{Z1} W.~Ziller, \emph{Examples of Riemannian manifolds with
    non-negative sectional curvature}, Metric and Comparison Geometry, Surv.\
    Diff.\ Geom.\ \textbf{11}, ed.\ K.~Grove and J.~Cheeger, International
    Press (2007)

\end{thebibliography}
\end{document}